\documentclass[lettersize,journal]{IEEEtran}
\usepackage{amsmath,amsfonts}
\usepackage{amsthm}
\usepackage{algorithm}
\usepackage{array}
\usepackage[caption=false,font=normalsize,labelfont=sf,textfont=sf]{subfig}
\usepackage{textcomp}
\usepackage{stfloats}
\usepackage{url}
\usepackage{verbatim}
\usepackage{graphicx}
\usepackage{mathtools}
\usepackage{cite}
\usepackage{booktabs}
\usepackage{multirow}
\usepackage{algpseudocode}
\usepackage{color}
\newtheorem{theorem}{Theorem}
\newtheorem{proposition}[theorem]{Proposition}

\newtheorem{definition}[theorem]{Definition}
\newtheorem{assumption}[theorem]{Assumption}

\newcommand{\E}{\mathbb{E}}
\newcommand{\Prob}{\mathbb{P}}

\begin{document}

\title{Strike Price Optimization for ISO New England's
Day-Ahead Ancillary Services}

\author{
Karl Zhu,
Parviz Alivand,
Jinye Zhao,
Tongxin Zheng,
and Dimitris Bertsimas
\thanks{Karl Zhu and Dimitris Bertsimas are with the Operations Research
Center, Massachusetts Institute of Technology, Cambridge, MA, USA.
Dimitris Bertsimas is also with the MIT Sloan School of Management.}
\thanks{Parviz Alivand, Jinye Zhao, and Tongxin Zheng are with ISO New
England, Holyoke, MA, USA. The views expressed in this paper are those of the authors and do not
represent the views of ISO New England.}
}

\markboth{Submitted to IEEE Transactions on Energy Markets, Policy and Regulation}%
{Zhu \MakeLowercase{\textit{et al.}}: Strike Price Optimization for
ISO New England's Day-Ahead Ancillary Services}

\maketitle

\begin{abstract}
ISO New England's (ISO-NE) Day-Ahead Ancillary Services Initiative settles
reserve products as financial call options on real-time energy prices.
The system-wide strike price creates an efficiency-reliability
tradeoff: {increasing it lowers competitive reserve offers, but weakens
resources' incentives to incur preparation costs and remain available
for real-time performance}. The existing strike price rule does not
explicitly account for heterogeneous resource incentives and reserve
requirements. We develop an optimization framework that selects the
highest strike price while ensuring that enough resources retain an
incentive to prepare and collectively satisfy the reserve requirements.
Because a resource's preparation decision may affect the resulting
real-time price distribution, its incentive depends on an unobservable
counterfactual. To address this, we derive a tight lower-bound
certificate using only the available conditional price distribution
and a bound on the resource's price impact. We show that each resource
enters the optimization through a single incentive threshold and that
any finite optimal strike price occurs at one of these thresholds.
This yields a tractable exact solution method based on threshold calculations and a small number of linear feasibility checks. Using reconstructed ISO-NE conditional price
distributions and representative gas-fired resources, we find that
higher heat-rate combustion turbines are more likely than
combined-cycle resources to constrain the strike price choice.
\end{abstract}

\section{Introduction}
\label{sec:introduction}
Operating reserves are essential for maintaining power-system
reliability. Before the implementation of the Day-Ahead Ancillary
Services Initiative (DASI), ISO New England (ISO-NE) relied on a combination of energy-only day-ahead market and out-of-market reliability unit commitment process to procure sufficient capacities to meet the next day's load forecast and operational reserve needs. Because the out-of-market reliability unit commitment process was not fully integrated into the day-ahead energy market, they limited transparent price formation
and the joint procurement of energy and reserve capability.

To address these shortcomings, ISO-NE implemented DASI on March 1, 2025 to integrate ancillary-service
procurement directly into the day-ahead market. Under DASI, energy is
co-optimized with four ancillary-service products:
Ten-Minute Spinning Reserve (TMSR), Ten-Minute Non-Spinning Reserve
(TMNSR), Thirty-Minute Operating Reserve (TMOR), and Energy Imbalance
Reserve (EIR). These products are settled as financial call options on
real-time energy prices. 

For each operating hour \(t\), ISO-NE announces a system-wide strike
price \(K_t\) before the day-ahead market clears. The same \(K_t\)
applies to all four ancillary-service products. A resource awarded one
MWh of reserve receives an upfront option premium and incurs the
real-time \textit{close-out} charge
\[
(P_t-K_t)^+,
\]
where \((x)^+:=\max\{x,0\}\) and \(P_t\) is the realized real-time energy price used for all four products. If the resource produces when \(P_t>K_t\), its
real-time energy revenue offsets the close-out charge, leaving a net
energy payment of \(K_t\) per MWh. Importantly, awarded resources are not physically obligated to produce energy. If they choose not to perform, they simply pay the
close-out charge and incur no additional penalties. As a result, the
strike price becomes the primary mechanism through which DASI incentivizes real-time performance \cite{iso_new_england_inc_revisions_2023}.

More fundamentally, the strike price governs a tradeoff between market efficiency and system reliability. A higher strike price reduces both the expected close-out exposure borne by reserve suppliers and the risk associated with that exposure, thereby lowering competitive reserve offers. However, the same reduction in close-out exposure weakens the financial benefit of preparing to produce in real time. Determining the strike price requires balancing these competing objectives.

Since DASI's implementation, ISO-NE has determined the strike price
using the rule
\[
K_t^{\mathrm{ISO}}
=
\E[P_t]+10,
\]
where the expectation is computed from a conditional Gaussian mixture
model (GMM) of the real-time price \cite{iso_new_england_inc_revisions_2023}. The rule is motivated by the
efficiency-reliability tradeoff. In a single-resource setting, setting
the strike price equal to the resource's short-run marginal cost gives
the highest strike price that preserves the maximum incentive to
prepare for real-time performance. Because no single marginal cost
represents a heterogeneous reserve fleet, the expected real-time price
provides a practical system-level reference, while the additional
\$10/MWh further reduces option exposure and competitive reserve
offers.

This construction is economically intuitive and operationally simple,
but it does not solve the underlying system-design problem. In
particular, \(\E[P_t]\) is not a direct estimate of the marginal costs
of the resources whose availability is required by the system. The
rule also does not explicitly account for resource-specific preparation
costs, reserve capabilities, product eligibility, or ISO-NE's nested
reserve structure. It is therefore best viewed as a practical
rule motivated by the underlying tradeoff, rather than as the
solution to a formal system-level optimization problem.

The importance of this design choice became apparent following the first year of market implementation. Prior to market launch, ISO-NE projected annual cost increases of approximately \$140 million, while the Internal Market Monitor subsequently reported incremental costs of \$974 million during the first 12 months of operation \cite{iso_new_england_internal_market_monitor_assessment_2026}. {Although these costs reflect multiple aspects of market design and market outcomes, the strike price emerged as an important parameter for reconsideration. ISO-NE subsequently proposed modifying the strike price rule by introducing a minimum floor constraint \cite{ewing_day-ahead_2026}. This development motivates a more explicit quantitative framework for relating the strike price to market efficiency and resource incentives for real-time availability.}

Building on the tradeoff underlying the current rule, this paper develops
an optimization framework for strike price design in ISO-NE's DASI
market. We formulate strike price selection as maximizing the
system-wide strike price while ensuring that sufficient resources retain
an incentive to prepare and collectively satisfy ISO-NE's reserve
requirements. The framework accommodates heterogeneous resource costs,
reserve capabilities, and product eligibility while preserving the
economic intuition of the current market design. Because some
resource-level parameters are not publicly observable, we view the
optimization model as a benchmark for the strike-price design problem,
with practical implementations approximating these quantities using
publicly or commercially available information.

Reliability options and related option-based market designs were
originally developed for long-run resource adequacy
\cite{vazquez_market_2002,oren_generation_2005,
bidwell_reliability_2005,cramton_forward_2008}.
Subsequent work examines their valuation, investment effects, and
regulatory implementation
\cite{andreis_pricing_2019,fontini_investing_2021,
mastropietro_reliability_2024,fabra_primer_2018}.
The strike price is a central design choice in this literature:
existing approaches link it to benchmark generation costs and fuel
indices
\cite{bidwell_reliability_2005,mastropietro_reliability_2024},
evaluate its effects on option value and investment
\cite{andreis_pricing_2019,fontini_investing_2021},
optimize the strike price within a market-equilibrium
model \cite{feng_generation_2024}, or prescribe it as a high quantile
of the electricity-price distribution \cite{roy_capoptix_2025}.
In the ISO-NE DASI setting, \cite{ent_risk-based_2026} studies how advance fuel procurement responds to a given strike price
for the EIR product. However, to our knowledge, no prior work has
explicitly optimized the DASI strike price.

Our contributions are as follows.
\begin{enumerate}

\item \textbf{Incentive characterization.}
We characterize a resource's incentive to incur an avoidable
preparation cost and remain available for real-time production.
Because preparation may affect the resulting real-time price
distribution, the incentive depends on an unobservable counterfactual.
We derive a tight lower-bound certificate that depends only on the available conditional price distribution and a bound on the resource's price
impact.

\item \textbf{Strike price optimization.}
We formulate system-wide strike price selection as maximizing the
strike price while requiring sufficient certified resources to satisfy
ISO-NE's nested reserve requirements. The model accounts for
heterogeneous preparation costs, marginal costs, reserve capabilities,
and product eligibility.

\item \textbf{Tractable threshold solution.}
We prove that each resource enters the strike price problem through a
single incentive threshold and that any finite optimum occurs at one
of these thresholds. This yields a tractable exact solution method
based on independent scalar threshold calculations and a finite number
of linear feasibility checks. We then quantify the resulting threshold
regimes for representative natural gas resources.
\end{enumerate}

The remainder of the paper is organized as follows.
Section~\ref{sec:tradeoff} formalizes the efficiency-reliability
tradeoff.
Section~\ref{sec:incentive} derives the incentive certificate and
characterizes its threshold structure.
Section~\ref{sec:optimization} formulates the strike price optimization model and develops an exact solution method.
Section~\ref{sec:empirical_thresholds} estimates these thresholds for
representative gas-fired resources and examines their sensitivity to
the modeling assumptions.
Section~\ref{sec:conclusion} concludes.

\section{Strike Price Tradeoff}\label{sec:tradeoff}
The strike price affects both reserve offers and incentives for
real-time availability. A higher strike price reduces reserve
suppliers' expected close-out exposure and therefore lowers competitive
offers. A lower strike price strengthens the incentive for awarded
resources to incur preparation costs and remain available when the
option is exercised. This section formalizes these opposing effects. 

{Throughout the analysis, let \(F_t\) denote the conditional
distribution of the real-time price \(P_t\), given the information
available when the strike price is determined. We treat \(F_t\) as
fixed with respect to the strike price \(K\); that is, the price
distribution is not conditioned on \(K\) itself. This is consistent
with ISO-NE's current methodology, which first estimates \(F_t\) using
a Gaussian mixture model and then computes the strike price from that
distribution
\cite{iso_new_england_market_development_statistical_2023}.}
\subsection{Efficiency: Higher Strike Prices Decrease Offer Prices}
\label{subsec:offer}
In a competitive market, a resource's offer reflects
the incremental cost of providing the reserve product. ISO-NE assumes the competitive offer price of the reserves can be decomposed into the following three components \cite{iso_new_england_market_development_competitive_2023}.

\begin{assumption}
\label{ass:offer_decomposition}
A competitive resource \(i\) in hour \(t\) submits an offer of the
form
\begin{equation*}
    \mathrm{Offer}_{i,t}(K)
    =
    \mathrm{AIC}_{i,t}
    +
    \E\left[(P_t-K)^+\right]
    +
    \rho_{i,t}(K),
\end{equation*}
where \(\mathrm{AIC}_{i,t}\) is the net incremental avoidable input cost of supporting the reserve obligation and is independent of \(K\), and \(\rho_{i,t}(K)\) is a risk premium that is decreasing in \(K\).
\end{assumption}

The avoidable input cost \(\mathrm{AIC}_{i,t}\) is the expected net
incremental cost of actions induced by a reserve award, measured relative
to the resource's next-best alternative. It reflects the cost of those
actions net of the expected economic value they create. If receiving the
reserve award does not change the resource's operating or procurement
decisions, \(\mathrm{AIC}_{i,t}\) may be zero. Otherwise, actions such as
advance fuel procurement or energy storage may give rise to a positive
\(\mathrm{AIC}_{i,t}\). In all cases, \(\mathrm{AIC}_{i,t}\) is
independent of $K$.

The risk premium \(\rho_{i,t}(K)\) being decreasing in \(K\) is natural. It compensates the resource for uncertainty associated with the option settlement. It may depend on
the probability that the option is exercised, \(\Prob(P_t>K)\), or
on the magnitude and tail risk of the close-out cost
\((P_t-K)^+\), for example through VaR or CVaR. Each of these
exposures is decreasing in \(K\). Thus, any risk premium that is
increasing in the resource's option exposure is decreasing in
$K$.

It follows that higher strike prices
decrease offer prices.

\begin{proposition}
\label{prop:offer_monotone}
Under Assumption~\ref{ass:offer_decomposition},
\(\mathrm{Offer}_{i,t}(K)\) is decreasing in \(K\).
\end{proposition}

\begin{proof}
For \(K_2 \ge K_1\),
\(
(P_t-K_2)^+ \le (P_t-K_1)^+
\)
for every realization of \(P_t\). Taking expectations preserves the
monotonicity. Since \(\mathrm{AIC}_{i,t}\) is independent of \(K\) and
\(\rho_{i,t}(K)\) is decreasing in \(K\), the result follows.
\end{proof}

The reduction in offer prices has two components. First, the expected
close-out payment decreases with the strike price. {Since the realized closeout charges} are collected and redistributed by the ISO, they are largely transfers
rather than social costs. Second, the resource's risk premium decreases
as its exposure to the option settlement falls. Higher strike prices
also reduce the frequency and magnitude of close-out settlements that
must be collected and redistributed. The efficiency gains therefore
come from lower risk premiums and administrative settlement costs.

\subsection{Reliability: Lower Strike Prices Increase Incentives}

The efficiency argument pushes the strike price upward. However,
reserve products are intended not only to reduce procurement costs,
but also to incentivize awarded resources to remain available for
real-time performance. As the strike price increases, the option is
exercised less frequently and the expected close-out exposure
declines, weakening the incentive to incur the preparation costs
required for availability. The central reliability question is
therefore whether the financial benefits of remaining available
outweigh these preparation costs.

We therefore characterize when an awarded resource has a financial
incentive to prepare for real-time availability. Consider a resource
\(i\) that has received a reserve award for hour \(t\). Let
\(c_{i,t}\) denote its short-run marginal cost of real-time
production. After receiving the award, the resource faces a binary
decision: it may incur an avoidable preparation cost \(A_{i,t}\)
and become available for real-time production, or avoid this cost
and remain unavailable. Unlike \(\mathrm{AIC}_{i,t}\), which is a net incremental cost used in offer formation, \(A_{i,t}\) is the underlying cost of preparation before accounting for the expected economic value created by being available. It may include fuel-procurement losses, startup costs, or other readiness expenditures. We assume that \(A_{i,t}\) does not depend on subsequent energy production and is sunk once the resource prepares.

A resource that becomes available acquires the option to supply energy in real time. If this additional supply is
deployed, it may reduce the resulting real-time price. We refer to
this effect as \emph{price impact}. To account for it, we
distinguish between two price distributions. Let \(P_t\) denote the baseline random real-time price with cumulative distribution function
\(F_t\), and let \(P^w_{i,t}\) denote the impacted real-time price that would occur if resource $i$ unilaterally prepared and became available, with
cumulative distribution function \(F^w_{i,t}\).

The expected payoff from remaining unavailable is
\begin{equation}
\pi^{wo}_{i,t}(K)
=
-
\E\left[(P_t-K)^+\right],
\label{eq:pi_wo}
\end{equation}
while the expected payoff from becoming available is
\begin{equation}
\pi^{w}_{i,t}(K)
=
\E\left[(P^{w}_{i,t}-c_{i,t})^+\right]
-
\E\left[(P^{w}_{i,t}-K)^+\right]
-
A_{i,t},
\label{eq:pi_w}
\end{equation}
where \(w\) and \(wo\) denote the cases with and without
preparation, respectively.

The resource prefers to prepare whenever the expected payoff from
becoming available exceeds that from remaining unavailable. The
incentive value of availability is therefore defined as follows.
\begin{definition}[Incentive value and incentive-compatibility]
\label{def:incentive}

The incentive value for resource $i$ in hour $t$ to be available at
strike price $K$ is
\begin{equation}
I_{i,t}(K)
:=
\pi^{w}_{i,t}(K)
-
\pi^{wo}_{i,t}(K).
\label{eq:incentive_def}
\end{equation}
Resource $i$ is \textit{incentive-compatible} in hour $t$ at $K$ if
\[
I_{i,t}(K)\ge0.
\]
\end{definition}
Substituting \eqref{eq:pi_wo} and \eqref{eq:pi_w} into \eqref{eq:incentive_def}:
\begin{align}
   I_{i,t}(K)
&=
\E[(P^w_{i,t}-c_{i,t})^+]
-
\E[(P^w_{i,t}-K)^+]
\notag\\
&\quad+
\E[(P_t-K)^+]
-
A_{i,t}.
\label{eq:incentive-explicit}
\end{align}
\section{Incentive Certificate Characterization}
\label{sec:incentive}
The central difficulty in estimating a resource's preparation incentive
is that its price impact is unobservable. In particular, the
counterfactual distribution \(F^w_{i,t}\) that would arise if the
resource became available cannot be observed directly from market data.
We address this problem by imposing economically interpretable bounds
on price impact and deriving a conservative incentive certificate that
depends only on the observed conditional price distribution. The
certificate is valid for every counterfactual distribution consistent
with those bounds. Moreover, it is tight, so no uniformly stronger
guarantee is possible without additional assumptions.

Throughout this section, we work with an arbitrary resource \(i\) and
hour \(t\) and suppress their indices when no ambiguity arises. All
distributions, expectations, and probabilities are understood to be
conditional on the information available when the strike price is
determined. Let \(P\sim F\) denote the observed real-time price and let
\(P_w\sim F_w\) denote the counterfactual price when the resource
prepares and becomes available. Define their survival functions by
\[
\bar F(p)=1-F(p),
\qquad
\bar F_w(p)=1-F_w(p).
\]

{
Below the resource's marginal cost, the resource would not be economically dispatched even if it were available. We therefore impose the following restriction.}
\begin{assumption}[No price impact below marginal cost]
\label{ass:no_impact}
Availability does not affect the price distribution below the
resource's marginal cost:
\[
F_w(p)=F(p),
\qquad p<c.
\]
\end{assumption}
{
Above the resource's marginal cost, additional availability may affect real-time prices because the resource can be economically dispatched. Rather than specifying a structural model of how the resource changes dispatch and price formation, we impose a reduced-form bound directly on the counterfactual price distribution.}
\begin{assumption}[Bounded price impact]
\label{ass:bounded_impact}
There exists a constant \(\phi\in[0,1]\) such that
\[
(1-\phi)\bar F(p)
\le
\bar F_w(p)
\le
\bar F(p),
\qquad p\ge c.
\]
\end{assumption}
{
Assumption~\ref{ass:bounded_impact} requires additional availability to weakly reduce the upper tail of the real-time price distribution while limiting the magnitude of this reduction. The upper bound \(\bar F_w(p)\le\bar F(p)\) requires availability not to increase the probability of prices exceeding any level \(p\ge c\). The lower bound limits the price impact: at each such price level, availability can reduce the exceedance probability by at most a fraction \(\phi\) of its baseline value. Thus, \(\phi=0\) corresponds to no price impact, while larger values of \(\phi\) permit progressively greater reductions in the upper tail. When restoring resource indices, we write this bound as \(\phi_i\), allowing price impact to vary across resources while assuming it is constant across hours. Importantly, the assumption does not specify how availability changes dispatch or individual price realizations; it only bounds the resulting change in the conditional price distribution.
}
Together, the two assumptions allow the unobservable counterfactual distribution to be eliminated from the incentive calculation. The following theorem gives a tight lower bound on the incentive value that holds for every counterfactual distribution satisfying these restrictions.
\begin{theorem}[Conservative incentive certificate]
\label{thm:conservative_incentive}
Under Assumptions~\ref{ass:no_impact} and
\ref{ass:bounded_impact}, the incentive value satisfies
\[
I(K)\ge \underline I(K),
\]
where
\begin{equation}
\underline I(K)
=
\begin{cases}
\E[(P-c)^+]-A,
& K\le c,\\[6pt]
\begin{aligned}[t]
&(1-\phi)\E[(P-c)^+]\\
&\quad+
\phi\E[(P-K)^+]-A,
\end{aligned}
& K>c.
\end{cases}
\label{eq:conservative_incentive}
\end{equation}
Moreover, the bound is tight: no uniformly stronger lower bound is possible.
\end{theorem}

\begin{proof}
First consider \(K\le c\). Applying the tail-integral identity
\[
\E[(Y-a)^+]
=
\int_a^\infty \Prob(Y>p)\,dp
\]
to \eqref{eq:incentive-explicit} gives
\begin{align*}
I(K)
&=
\int_c^\infty \bar F_w(p)\,dp
-
\int_K^\infty \bar F_w(p)\,dp
+
\int_K^\infty \bar F(p)\,dp
-
A\\
&=
-\int_K^c \bar F_w(p)\,dp
+
\int_K^c \bar F(p)\,dp
+
\int_c^\infty \bar F(p)\,dp
-
A.
\end{align*}
By Assumption~\ref{ass:no_impact},
\(\bar F_w(p)=\bar F(p)\) for \(p<c\). The first two
integrals therefore cancel, yielding
\[
I(K)
=
\int_c^\infty \bar F(p)\,dp-A
=
\E[(P-c)^+]-A.
\]
Thus, the certificate is exact for \(K\le c\).

Now consider \(K>c\). The same identity gives
\begin{align*}
I(K)
&=
\int_c^\infty \bar F_w(p)\,dp
-
\int_K^\infty \bar F_w(p)\,dp
+
\E[(P-K)^+]
-
A\\
&=
\int_c^K \bar F_w(p)\,dp
+
\E[(P-K)^+]
-
A.
\end{align*}
By Assumption~\ref{ass:bounded_impact},
\(
\bar F_w(p)\ge(1-\phi)\bar F(p)
\) for $p \ge c$. Consequently,
\begin{align*}
I(K)
&\ge
(1-\phi)\int_c^K \bar F(p)\,dp
+
\E[(P-K)^+]
-
A\\
&=
(1-\phi)\E[(P-c)^+]
-
(1-\phi)\E[(P-K)^+]
\notag\\
&\quad+
\E[(P-K)^+]
-
A\\
&=
(1-\phi)\E[(P-c)^+]
+
\phi\E[(P-K)^+]
-
A.
\end{align*}

It remains to establish tightness. For \(K\le c\), equality holds for
every pair \((F,F_w)\) satisfying
Assumption~\ref{ass:no_impact}. For \(K>c\), consider the
counterfactual distribution
\[
F_w(p)
=
\begin{cases}
F(p),
& p<c,\\
F(p)+\phi\bar F(p),
& p\ge c.
\end{cases}
\]
This distribution satisfies Assumptions~\ref{ass:no_impact} and
\ref{ass:bounded_impact}, and it has
\[
\bar F_w(p)=(1-\phi)\bar F(p),
\qquad p\ge c.
\]
Hence every inequality in the preceding derivation holds with
equality, proving that the bound is tight.
\end{proof}

Theorem~\ref{thm:conservative_incentive} shows that incentive
compatibility depends only on the upper tail of the conditional
real-time price distribution. For \(K>c\), the conservative
certificate depends on the tail expectations
\(\E[(P-c)^+]\) and \(\E[(P-K)^+]\), together with the
preparation cost \(A\). Thus, strike price design is fundamentally
an upper-tail problem: the availability incentives created by
reserve products are determined by scarcity events rather than
typical price outcomes.
\begin{proposition}[Threshold structure of the conservative certificate]
\label{prop:shape}

For any resource \(i\) and hour \(t\), suppose
\(c_{i,t}\ge0\), \(\phi_{i}\in[0,1]\), and the density \(f_t\) is
strictly positive above \(c_{i,t}\). Then
\(\underline I_{i,t}(K)\) is constant for
\(K\le c_{i,t}\) and weakly decreasing for
\(K>c_{i,t}\). Furthermore, exactly one of the following holds:

\begin{enumerate}
    \item
    \[
    \underline I_{i,t}(K)<0
    \qquad \forall\,K\ge0.
    \]
    The certificate does not establish incentive compatibility at
    any strike price.

    \item
    \[
    \underline I_{i,t}(K)\ge0
    \qquad \forall\,K\ge0.
    \]
    The resource is incentive-compatible at every strike price.

    \item There exists a unique threshold
    \(\kappa_{i,t}\ge c_{i,t}\) satisfying
    \[
    \underline I_{i,t}(\kappa_{i,t})=0.
    \]
    The resource is incentive-compatible for
    \(K\le\kappa_{i,t}\). This case can occur only if \(\phi_i>0\).
\end{enumerate}
\end{proposition}
\begin{proof}
Suppress the resource and time indices. Differentiating
\eqref{eq:conservative_incentive} with respect to \(K\) gives
\[
\frac{d}{dK}\underline I(K)
=
\begin{cases}
0,
& K<c,\\[6pt]
-\phi\Prob(P>K),
& K>c.
\end{cases}
\]

Thus, \(\underline I(K)\) is constant for \(K\le c\) and weakly
decreasing for \(K>c\). If \(\phi=0\), the certificate is constant
for all \(K\ge0\), so either case 1 or case 2 holds.

Now suppose \(\phi>0\). Since the density is strictly positive above
\(c\),
\[
\Prob(P>K)>0
\qquad
\text{for every finite }K>c,
\]
and hence \(\underline I(K)\) is strictly decreasing above \(c\).
By continuity, it must therefore remain negative for every \(K\),
remain nonnegative for every \(K\), or cross zero at a unique
\(\kappa\in[c,\infty)\). These correspond to cases 1--3,
respectively.

\end{proof}
Figure~\ref{fig:single_generator_incentive} illustrates the unique
threshold case in Proposition~\ref{prop:shape} for an example resource
and hour. The illustrative conditional real-time price distribution is
\(
P
\sim
0.5251\,\mathcal N(34.66,6.54^2)
+
0.4270\,\mathcal N(55.76,17.52^2)
+
0.0479\,\mathcal N(104.86,71.29^2).
\)
This distribution corresponds to the ISO-NE conditional GMM for July 1, 2025 at 12:00~p.m.; see \cite{iso_new_england_market_development_statistical_2023} for how the GMM is fitted. For an illustrative resource
with \(c=40\), \(A=10\), and \(\phi=0.2\), the upper panel shows the
conditional real-time price density, while the lower panel plots the
conservative certificate \(\underline I(K)\). The certificate is
constant for \(K\le c\) and strictly decreasing for \(K>c\), crossing
zero at \(\kappa=60.8\). Thus, the certificate establishes incentive
compatibility for \(K\le60.8\), while it is inconclusive for larger
strike prices. Changing the preparation cost \(A\) shifts the
certificate vertically, producing the always-negative and
always-nonnegative cases in Proposition~\ref{prop:shape}.
\begin{figure}[t]
    \centering
    \includegraphics[width=0.9\linewidth]
    {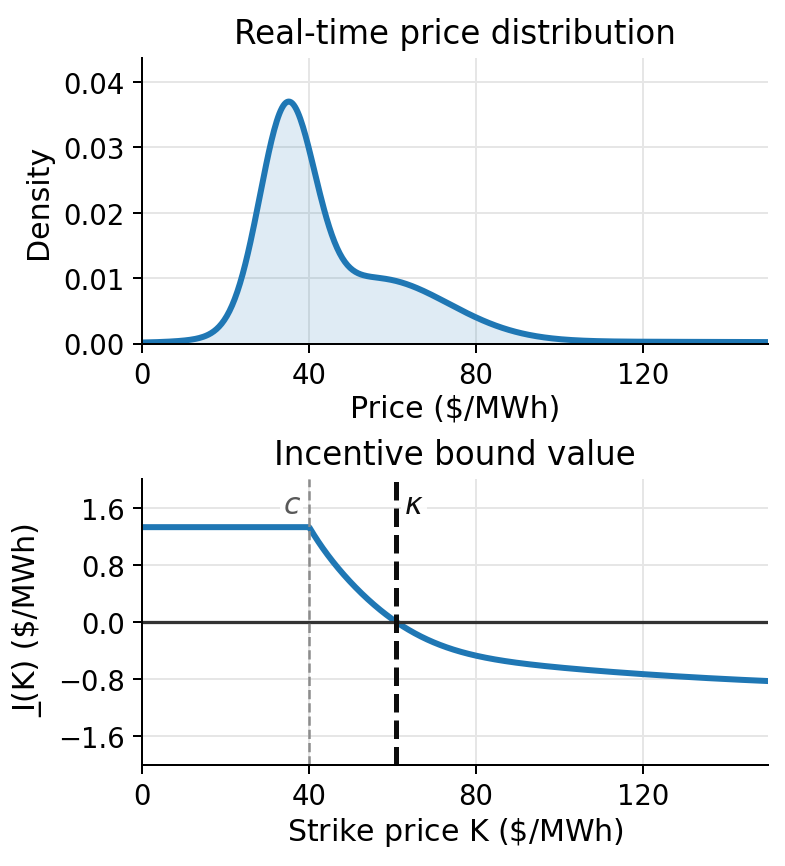}
\caption{Threshold structure of the conservative incentive certificate
for an illustrative resource \(i\) and hour \(t\). The conditional
real-time price distribution corresponds to July 1, 2025 at
12:00~p.m. With \(c=40\), \(A=10\), and \(\phi=0.2\), the certificate
is constant below \(c\), decreases above \(c\), and crosses zero at
\(\kappa=60.8\).}
    \label{fig:single_generator_incentive}
\end{figure}
\section{Strike Price Optimization}
\label{sec:optimization}
We formulate strike price selection as maximizing the strike price
while ensuring that enough resources are incentive-compatible to meet ISO-NE's reserve requirements.
Proposition~\ref{prop:offer_monotone} motivates the objective, while
the threshold structure developed in Section~\ref{sec:incentive}
yields an exact finite solution method.
\subsection{Model Formulation}
\label{subsec:model}
\begin{table}[t]
\caption{Sets, parameters, and decision variables.}
\label{tab:model}
\centering
\footnotesize
\renewcommand{\arraystretch}{1.1}
\setlength{\tabcolsep}{4pt}

\begin{tabular}{p{2.0cm} p{5.7cm}}
\toprule
\textbf{Symbol} & \textbf{Description} \\
\midrule

\multicolumn{2}{l}{\textbf{Sets}} \\
\midrule

$\mathcal{G}$ &
Reserve-eligible resources. \\

$\mathcal{Q}$ &
Reserve requirement categories. \\

$\mathcal{A}$ &
Reserve products. \\

$\mathcal{A}_q \subseteq \mathcal{A}$ &
Reserve products eligible to satisfy
requirement \(q\in\mathcal{Q}\). \\

$\mathcal{A}_i \subseteq \mathcal{A}$ &
Reserve products that resource
\(i\in\mathcal{G}\) is eligible to provide. \\

\midrule

\multicolumn{2}{l}{\textbf{Parameters}} \\
\midrule

$F_t, f_t$ &
Conditional distribution and density of the
real-time price \(P_t\) \\
$c_{i,t}$ &
Short-run marginal cost of real-time production
for resource \(i\). \\

$A_{i,t}$ &
Fixed preparation cost of resource \(i\). \\

$\phi_{i}$ &
Bound on the price impact of resource \(i\). \\

$\overline{R}_{i,a,t}$ &
Maximum capability (MW) of resource \(i\)
for reserve product \(a\). \\

$\overline{R}_{i,30,t}$ &
Maximum cumulative 30-minute reserve capability
(MW) of resource \(i\). \\

$\mathrm{Req}_{q,t}$ &
Reserve requirement (MW) for category
\(q\in\mathcal{Q}\). \\

\midrule

\multicolumn{2}{l}{\textbf{Decision Variables}} \\
\midrule

$K_t$ &
Strike price for hour \(t\). \\

$z_{i,t}$ &
Binary indicator equal to 1 if resource \(i\)
is counted as certified incentive-compatible. \\

$r_{i,a,t}$ &
Auxiliary allocation (MW) from resource \(i\)
to product \(a\), used to verify reserve coverage. \\

\bottomrule
\end{tabular}
\end{table}
Table~\ref{tab:model} defines the sets, parameters, and decision variables. ISO-NE's nested reserve structure is \(
\mathcal{A}_{\text{TMSR}}=\{\text{TMSR}\}, \
\mathcal{A}_{\mathrm{Total10}}=\{\text{TMSR, TMNSR}\}, \
\mathcal{A}_{\mathrm{Total30}}=\{\text{TMSR, TMNSR, TMOR}\}.
\)
For each hour \(t\), the model is
\begin{subequations}
\label{eq:main_model}
\begin{align}
\max_{K_t,\,\boldsymbol{r}_t,\,\boldsymbol{z}_t} \quad
    & K_t \label{eq:obj} \\
\text{s.t.} \quad
    & z_{i,t}=1
    \;\implies\;
    \underline I_{i,t}(K_t)\ge0,
    && \forall i\in\mathcal G, \label{eq:ic} \\
    & \boldsymbol r_t
    \in
    \mathcal R_t(\boldsymbol z_t),
    \label{eq:feas} \\
    & K_t\ge0, \label{eq:bounds} \\
    & z_{i,t}\in\{0,1\},
    && \forall i\in\mathcal G. \label{eq:binary}
\end{align}
\end{subequations}

Objective~\eqref{eq:obj} selects the highest strike price.
Constraint~\eqref{eq:ic} allows a resource to count toward reserve
coverage only if it is certified incentive-compatible.
Constraint~\eqref{eq:feas} requires the counted resources to satisfy
all reserve requirements, where
\begin{equation}
\label{eq:R_set}
\begin{alignedat}{2}
\mathcal R_t(\boldsymbol z_t)
=\{\boldsymbol r_t:~&
\sum_{i\in\mathcal G}
\sum_{a\in\mathcal A_q\cap\mathcal A_i}
r_{i,a,t}
\ge \mathrm{Req}_{q,t}, \quad \forall q\in\mathcal Q,\\
&0\le r_{i,a,t}
\le \overline R_{i,a,t}z_{i,t}, \quad \forall i\in\mathcal G,\ a\in\mathcal A_i,\\
&\sum_{a\in\mathcal A_{\mathrm{Total30}}\cap\mathcal A_i}
r_{i,a,t}
\le \overline R_{i,30,t}z_{i,t}, \quad \forall i\in\mathcal G\}.
\end{alignedat}
\end{equation}
\subsection{Solution Method}
\label{subsec:structure}

Although problem~\eqref{eq:main_model} is a
mixed-integer nonlinear optimization model, it exhibits a special optimal solution structure. Proposition~\ref{prop:shape} shows that each resource is
certified incentive-compatible up to an upper threshold
\(\kappa_{i,t}\). Hence, the set of resources that can be counted
toward reserve coverage changes only when \(K_t\) crosses one of
these thresholds. Therefore, Theorem~\ref{thm:endpoint} reduces the optimization to a finite search over the threshold values.
\begin{theorem}[Threshold optimality]
\label{thm:endpoint}
If problem~\eqref{eq:main_model} is feasible and bounded, then there
exists an optimal strike price \(K_t^*\) satisfying
\(
K_t^*=\kappa_{i,t}
\) for some \(i\in\mathcal G\) with \(\kappa_{i,t}<\infty\).
\end{theorem}

\begin{proof}
For a given strike price \(K_t\), define
\[
\widehat z_{i,t}(K_t)
=
\mathbf 1\!\left\{
\underline I_{i,t}(K_t)\ge0
\right\},
\qquad i\in\mathcal G.
\]
Thus, \(\widehat{\boldsymbol z}_t(K_t)\) includes every resource
certified incentive-compatible at \(K_t\). It is sufficient to test
reserve feasibility using this vector: if any subset of the certified
resources can satisfy the reserve requirements, then including all
certified resources preserves feasibility.

By Proposition~\ref{prop:shape},
\(\widehat{\boldsymbol z}_t(K_t)\) changes only when \(K_t\) crosses
a resource threshold. Suppose an optimal strike price \(K_t^*\) is not
a finite resource threshold. If there is a smallest finite threshold
\(\bar\kappa>K_t^*\), then the set of certified resources is unchanged
on \([K_t^*,\bar\kappa]\). Hence, \(\bar\kappa\) is also feasible,
contradicting the optimality of \(K_t^*\). If no such threshold exists,
the set of certified resources remains unchanged for every
\(K_t>K_t^*\), so the strike price can be increased without losing
feasibility, contradicting boundedness. Therefore, an optimal strike
price is attained at a finite resource threshold.
\end{proof}

The threshold structure also characterizes the infeasible and
unbounded cases. Since the set of certified resources can only shrink
as \(K_t\) increases, the problem is infeasible if the resources
certified at \(K_t=0\) cannot satisfy the reserve requirements.
Conversely, if the resources certified for every strike price can
satisfy the requirements, then the same allocation remains feasible
for arbitrarily large \(K_t\), and the problem is unbounded. In all
other feasible cases, the optimum is finite and
Theorem~\ref{thm:endpoint} applies.

Algorithm~\ref{alg:endpoint} summarizes the resulting exact solution
method. Each resource threshold is computed independently using a
one-dimensional root-finding method. The algorithm checks the
unbounded case and then evaluates the distinct finite thresholds in
decreasing order. It requires at most \(|\mathcal G|+1\) linear
reserve-feasibility checks, with all nonlinear computation confined to
the scalar threshold calculations.
\begin{algorithm}[t]
\caption{Threshold enumeration for hour \(t\)}
\label{alg:endpoint}
\begin{algorithmic}[1]
\For{each \(i\in\mathcal G\)}
    \If{\(\underline I_{i,t}(0)<0\)}
        \State Mark \(i\) as never certified.
    \ElsIf{\(\displaystyle
        \lim_{K\to\infty}\underline I_{i,t}(K)\ge0\)}
        \State Mark \(i\) as always certified; set
        \(\kappa_{i,t}=\infty\).
    \Else
        \State Compute \(\kappa_{i,t}\).
    \EndIf
\EndFor

\State Let \(\boldsymbol z_t^\infty\) indicate always-certified resources.
\If{\(\mathcal R_t(\boldsymbol z_t^\infty)\neq\emptyset\)}
    \State \Return unbounded
\EndIf

\For{each finite \(\kappa_{i,t}\) in decreasing order}
    \State Set \(K_t\gets\kappa_{i,t}\).
    \State Set
    \(\widehat z_{i,t}
    =\mathbf 1\{\underline I_{i,t}(K_t)\ge0\}\),
    \(i\in\mathcal G\).
    \If{\(\mathcal R_t(\widehat{\boldsymbol z}_t)\neq\emptyset\)}
        \State \Return \(K_t^*\gets K_t\)
    \EndIf
\EndFor

\State \Return infeasible
\end{algorithmic}
\end{algorithm}
\subsection{Illustrative Strike Price Calculation}
\label{subsec:example2}

We illustrate Algorithm~\ref{alg:endpoint} using three natural-gas
resources:
\[
\mathcal G
=
\{\mathrm{GasCC1},\mathrm{GasCC2},\mathrm{GasCT}\},
\]
consisting of two combined-cycle units and one combustion turbine.
We use the same conditional real-time price distribution as in
Figure~\ref{fig:single_generator_incentive}, corresponding to
July 1, 2025 at 12:00~p.m., together with the price-impact bound
\(\phi=0.2\). Their \((c_i,A_i)\) values are
\((35,5)\), \((37,7)\), and \((40,10)\), respectively.

For this single-hour aggregate example, we suppress the product and
time indices. Table~\ref{tab:algo_trace} reports each resource's reserve
capability \(\overline R_i\) and threshold \(\kappa_i\), obtained by
solving
\(
\underline I_i(\kappa_i)=0.
\) For simplicity, consider a single aggregate reserve requirement of
\(1000\) MW; formulation~\eqref{eq:R_set} accommodates the full nested
reserve structure. At \(K=145.4\), only Gas CC1 is certified, providing
600 MW and leaving the requirement unmet. At \(K=64.6\), Gas CC2 is
also certified, increasing the available capability to 1100 MW.
Therefore,
\(
K^*=\$64.6/\mathrm{MWh}.
\)
Consistent with Theorem~\ref{thm:endpoint}, the optimal strike price is
attained at a resource threshold.

We deliberately omit a system-wide empirical implementation because a
credible calculation would require a separate, detailed study of the
eligible fleet, including hourly marginal costs, avoidable preparation
costs, reserve capabilities, product eligibility, and resource-specific
price impacts. In particular, the preparation costs and price-impact
bounds require careful operational and market analysis. A fleet-wide
result based on simplified assumptions would not be informative. Instead, the next section focuses on the resource-level incentive thresholds that serve as inputs to the optimization.
\begin{table}[t]
\centering
\caption{Illustration of threshold enumeration with an aggregate
reserve requirement of \(1000\) MW.}
\label{tab:algo_trace}
\footnotesize
\begin{tabular}{lrr}
\toprule
Resource & Capability (MW) & Threshold (\$/MWh) \\
\midrule
Gas CC1 & 600 & 145.4 \\
Gas CC2 & 500 & 64.6 \\
Gas CT  & 400 & 46.5 \\
\midrule
\multicolumn{3}{l}{\emph{Threshold enumeration}}\\
\midrule
Candidate \(K\) & Certified capacity & Status \\
\midrule
145.4 & 600 MW  & Infeasible \\
\textbf{64.6} & \textbf{1100 MW} & \textbf{Optimal} \\
46.5  & 1500 MW & Feasible \\
\bottomrule
\end{tabular}
\end{table}

\section{Empirical Study of Incentive Thresholds}
\label{sec:empirical_thresholds}
We estimate hourly incentive thresholds \(\kappa_{i,t}\) for
representative gas-fired resources. Specifically, we consider a
combined-cycle resource (Gas CC) and a combustion turbine (Gas CT).
These resources are relevant because preparing for real-time
availability may require committing to fuel before the operating hour.
We report threshold statistics under a baseline specification and
examine their sensitivity to the cost and price-impact assumptions.

For the real-time price distributions, we reconstruct the first
12 months of ISO-NE's conditional Gaussian mixture models, covering
March 1, 2025 through February 28, 2026. The reconstruction uses
ISO-NE's public methodological memorandum \cite{iso_new_england_market_development_statistical_2023} together with the model features available internally.

For a gas-fired resource \(i\), let \(HR_i\) denote its heat rate in
MMBtu/MWh. We use representative values
\(HR_{\mathrm{GasCC}}=7\) and \(HR_{\mathrm{GasCT}}=11\), consistent
with the values used in the Internal Market Monitor revisions \cite{ewing_day-ahead_2026}. Let \(g_t^{DA}\) denote the
Algonquin Citygate day-ahead natural gas price in \$/MMBtu. This index
is also an input to ISO-NE's conditional GMM.

Preparing the resource requires committing to fuel before the operating
hour. If the resource does not produce, part of the committed fuel value
may be recovered through resale, rebalancing, or adjustment of its fuel
position. Let \(\alpha\in[0,1]\) denote the recoverable fraction of the
fuel value. We set
\[
c_{i,t}=\alpha HR_i g_t^{DA},
\qquad
A_{i,t}=(1-\alpha)HR_i g_t^{DA}.
\]
Thus, \(c_{i,t}\) represents the recoverable portion of the fuel value
and \(A_{i,t}\) the irreversible portion incurred upon preparation. The
decomposition satisfies
\(c_{i,t}+A_{i,t}=HR_i g_t^{DA}\), so the full fuel commitment is
allocated between the marginal production cost and the fixed preparation
cost without double counting.

The baseline specification uses \(\alpha=0.8\), corresponding to an
80\% recoverable fuel value and a 20\% irreversible commitment loss.
Because \(\alpha\) is a reduced-form assumption rather than an
observed resource-specific recovery rate, we also examine alternative
values in the sensitivity analysis.

By Proposition~\ref{prop:shape}, each resource \(i\) and hour \(t\)
falls into one of
three regimes. The certificate may be negative for every strike price
(never certified), remain nonnegative for every strike price (always
certified), or cross zero at a unique finite threshold satisfying
\(\underline I_{i,t}(\kappa_{i,t})=0\).
Table~\ref{tab:kappa_distribution} reports the threshold distribution
under the baseline specification
\((\alpha,\phi)=(0.8,0.2)\). For Gas CC, the certificate is negative
for every strike price in 3.9\% of hours and remains nonnegative for
every strike price in 95.1\%, leaving finite thresholds in 1.0\% of
hours. For Gas CT, the corresponding shares are 31.1\%, 60.6\%, and
8.3\%.
\begin{figure*}[t]
    \centering
    \includegraphics[width=\linewidth]
    {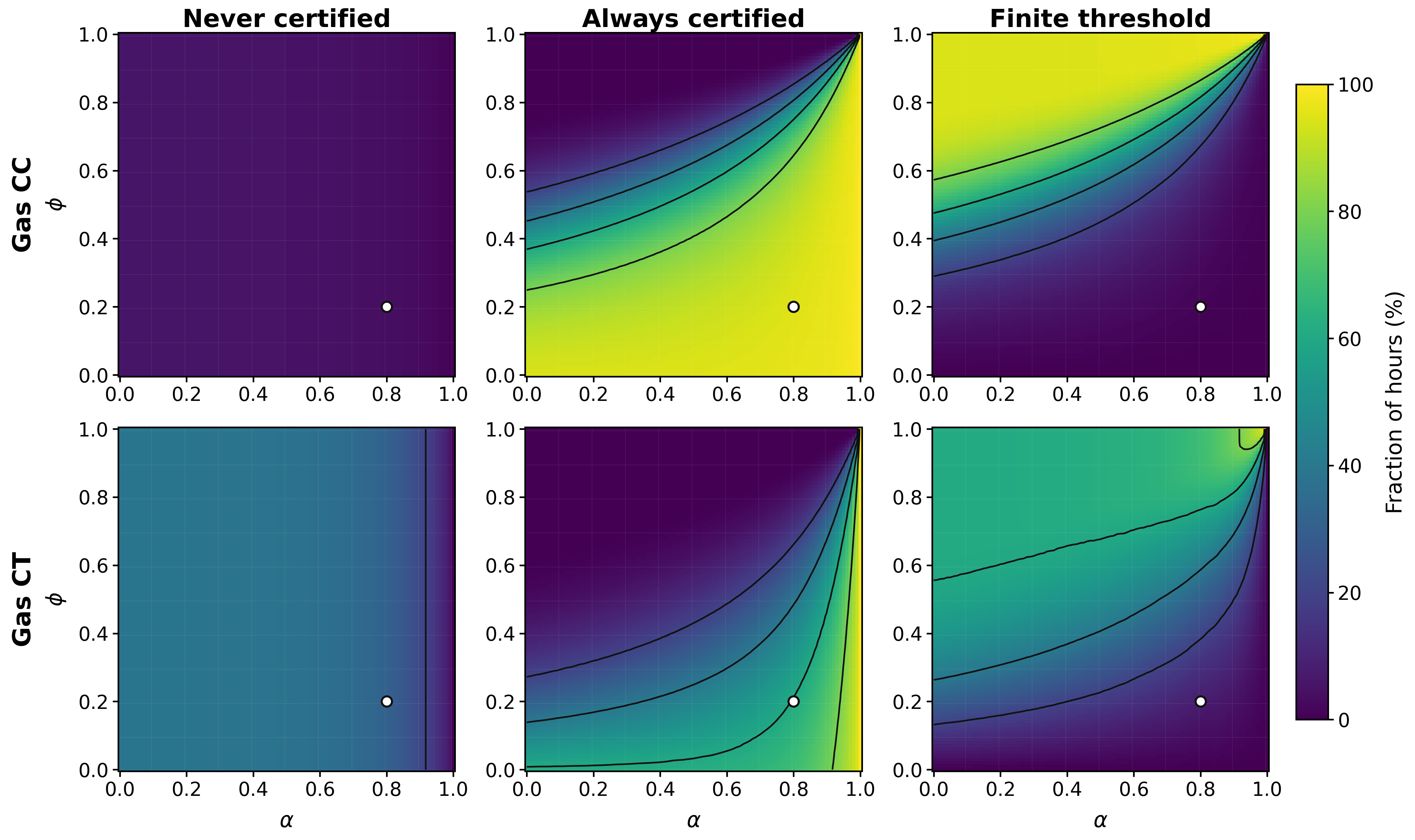}
    \caption{Fraction of hours in the never-certified, always-certified, and finite-threshold regimes for representative Gas CC and Gas CT
    resources under alternative salvage factors \(\alpha\) and
    price-impact bounds \(\phi\). Open circles indicate the baseline
    specification \((\alpha,\phi)=(0.8,0.2)\).}
    \label{fig:threshold_regimes}
\end{figure*}
Among hours with finite thresholds, the median threshold is
\$159.6/MWh for Gas CC and \$44.8/MWh for Gas CT. The lower Gas CT
thresholds and higher frequency of certification failure reflect its
higher heat rate and therefore larger fuel commitment for each MWh of
production.

These classifications identify when a resource is responsive to the
strike price. A resource in the never-certified regime cannot be counted toward reserve coverage under the conservative certificate, regardless of
the strike price. Only resources with
finite thresholds can directly constrain the strike price choice.

The always-certified regime should not, however, be interpreted as
implying that the system-wide problem is unbounded. Unboundedness
occurs only if the always-certified resources, taken together and
subject to their product eligibility and capability limits, can satisfy
every reserve requirement. If additional resources with finite thresholds are needed to satisfy
any reserve requirement, then those thresholds impose a finite upper
bound on the system-wide strike price.

Distinguishing these regimes is therefore useful for identifying which
resources and hours are likely to matter for strike price design.
Doing so in practice would require predicting the regime from the
hourly price distribution and resource characteristics. The
never-certified classification remains conservative: it indicates that
the incentive bound is negative --- the actual incentive is not necessarily negative.
\begin{table}[t]
\centering
\caption{Distribution of incentive thresholds under the baseline
specification. Monetary entries are in \$/MWh and are computed only
over finite thresholds.}
\label{tab:kappa_distribution}
\footnotesize
\begin{tabular}{lrrrrr}
\toprule
Resource & Median & P90 & P99 & Never certified &
\(\kappa=\infty\) \\
\midrule
Gas CC & 159.6 & 261.5 & 354.5 & 3.9\% & 95.1\% \\
Gas CT & 44.8  & 156.8 & 289.3 & 31.1\% & 60.6\% \\
\bottomrule
\end{tabular}
\end{table}

The threshold regime depends on both the salvage fraction \(\alpha\)
and the assumed price impact \(\phi\).
Figure~\ref{fig:threshold_regimes} reports the fraction of hours in
which each representative resource is never certified, has a finite
threshold, or remains certified for every strike price. Changing
\(\alpha\) reallocates the fuel commitment between the short-run
marginal cost \(c_{i,t}\) and the preparation cost \(A_{i,t}\), so its
net effect is determined by both components. Increasing \(\phi\)
weakens the conservative certificate, shifting hours from the
always-certified regime to the finite-threshold regime. The
never-certified share is unaffected by \(\phi\), because certification
already fails at \(K\le c_{i,t}\) in those hours.

The finite-threshold panels identify the parameter regions in which a
resource can impose a finite upper bound on the strike price. The
figure also shows that Gas CT is substantially more likely than Gas CC
to be never certified or to have a finite threshold. Under the assumed fuel-cost construction, the higher heat-rate Gas CT
is more likely to restrict the set of strike prices that can be
certified.

\section{Conclusion}
\label{sec:conclusion}

This paper develops an optimization framework for selecting the
system-wide strike price in ISO-NE's Day-Ahead Ancillary Services
market. We first characterize the incentive of a resource to incur
preparation costs and to remain available for real-time performance.
Because this incentive depends on the counterfactual price distribution
that would arise if the resource became available, we derive a tight
conservative certificate using bounds on the resource's price impact.
The certificate is constant below the resource's marginal cost and
decreases thereafter, yielding a simple threshold structure.

We use this structure to formulate strike price selection as maximizing
the strike price while requiring sufficient certified resources to
satisfy ISO-NE's reserve requirements. We show that any finite optimum
is attained at a resource threshold, which reduces the resulting
mixed-integer nonlinear formulation to scalar threshold calculations
and a finite sequence of linear reserve-feasibility checks.

The empirical study illustrates how these thresholds vary between
representative gas-fired resources and modeling assumptions. Under the
baseline specification, the higher-heat-rate Gas CT is more likely than
the Gas CC to be never certified or to impose a finite upper bound on
the strike price. More generally, the never-certified, finite-threshold,
and always-certified regimes identify which resources are responsive to
changes in the strike price and which are likely to constrain the
system-wide decision.

A complete system-wide implementation would require a detailed
resource-level study of preparation costs, reserve capabilities,
product eligibility, and price impacts. The present analysis therefore
provides the optimization framework and the resource-level quantities
needed for such an implementation, rather than a calibrated estimate of
ISO-NE's optimal strike price. Future work may combine the framework
with detailed fleet data and market simulations to evaluate 
strike prices and market outcomes throughout the system.
\bibliographystyle{ieeetr}
\bibliography{references}


\end{document}